\documentclass[10pt,a4paper]{amsart}
\usepackage{amsthm} 
\usepackage{amsfonts}
\usepackage{amsmath}
\usepackage{amssymb}
\usepackage{dsfont}
\usepackage{txfonts}
\usepackage{cancel}
\usepackage{soul}
\usepackage{mathrsfs}
\usepackage{enumerate}
\usepackage{textcomp}
\usepackage{hyperref} 
\usepackage{color}
\usepackage[dvipsnames]{xcolor}

 \title[On nonnegative invariant quartics in type $A$]
{On nonnegative invariant quartics in type $A$
}
\author{Sebastian Debus}
\address{Faculty of Mathematics, Technische Universität Chemnitz, Germany}
\email{sebastian.debus@mathematik.tu-chemnitz.de}
\author{Charu Goel}
\address{Department of Basic Sciences, Indian Institute of Information Technology Nagpur, India}
\email{charugoel@iiitn.ac.in}
\author{Salma Kuhlmann}
\address{Department of Mathematics and Statistics, University of Konstanz, Germany}
\email{salma.kuhlmann@uni-konstanz.de}
\author{Cordian Riener}
\address{Department of Mathematics and Statistics, UIT The Arctic University of Norway, Norway}
\email{cordian.riener@uit.no}

\date{}

\numberwithin{equation}{section}

\begin{document}
\maketitle

\theoremstyle{definition}
\numberwithin{equation}{section}
\newtheorem{thm}{Theorem}[section]
\newtheorem{prop}[thm]{Proposition}
\newtheorem{lemma}[thm]{Lemma}
\newtheorem{cor}[thm]{Corollary}
\newtheorem{con}{Question}
\newtheorem{notation}{Notation}
\newtheorem{conj}{Conjecture}[section]
\newtheorem{definition}[thm]{Definition}
\newtheorem{remark}[thm]{Remark}
\theoremstyle{definition}
\newtheorem{ex}[thm]{Example}
\newcommand{\commentit}[1]{\noindent {\color{RED} $\RHD$ \textsc{#1} $\LHD$}}
\newenvironment{cmtlist}
 {\begin{itemize}\color{RED}\sc }
 {\end{itemize}}
\definecolor{RED}{rgb}{0.6,0,0}
\setlength{\parskip}{0.5em}
 \newcommand{\N}{\mathbb{N}}
\newcommand{\PP}{\mathcal{P}}
\newcommand{\R}{\mathbb{R}}
\newcommand{\C}{\mathbb{C}}
\newcommand{\Q}{\mathbb{Q}}
\newcommand{\Sn}{\mathcal{S}_n}
\newcommand\g{\mathfrak{g}}
\newcommand{\al}{\alpha}
\newcommand{\be}{\beta}
\newcommand{\ga}{\gamma}
\newcommand{\la}{\lambda}


\vspace{-1cm}

\begin{abstract}
The equivariant nonnegativity versus sums of squares question has been solved for any infinite series of essential reflection groups but type $A$. As a first step to a classification, we analyse $A_n$-invariant quartics. We prove that the cones of invariant sums of squares and nonnegative forms are equal if and only if the number of variables is at most $3$ or odd.
\end{abstract}

 \section{Introduction}
The study of \emph{nonnegative} real polynomials, i.e. polynomials whose evaluation on any point is nonnegative, is a topic of interest from many perspectives, e.g. verificiation of polynomial inequalities and polynomial optimization. From complexity theoretical view the verification is NP-hard \cite{blum1998complexity}. If one can write a real polynomial as a \emph{sum of squares} of real polynomials, then the polynomial is clearly nonnegative. It was shown by Hilbert \cite{Hilb_1} in his celebrated theorem from 1888 that there are basically three cases where any nonnegative polynomial is a sum of squares. We formulate Hilbert's theorem in terms of \emph{forms}, i.e. homogeneous polynomials, since any polynomial is nonnegative if and only if its homogenization is nonnegative and a sum of squares if and only if its homogenization is a sum of squares \cite{Marshall}.
Hilbert showed that the cones of nonnegative forms and that of sums of squares of degree $2d$ in $n$ variables are equal if and only if $(n,2d) \in \{(2,2d'),(n',2),(3,4) \ | \ n',d' \in \N \}$. Hilbert's proof was unconstructive and it took almost 80 years until the first example of a nonnegative polynomial which is not a sum of squares was given (this is the Motzkin polynomial \cite{Motz-1}). It was then asked by Hilbert whether any nonnegative polynomial is a sum of squares of rational functions. This is known as Hilbert's 17th problem. E. Artin 
proved that this is true, thereby lying the cornerstone of the field of \emph{real algebraic geometry}.

Motivated by Hilbert's 1888 theorem, several authors investigated the equivariant setting. For a group $G$ acting on the real polynomial ring one restricts to \emph{invariant} forms, i.e. forms which are fixed under the action of $G$. Choi, Lam and Reznick investigated the question for the symmetric group $S_n$ which was completed by Goel, Kuhlmann and Reznick \cite{G-K-R}. The signed symmetric group $B_n$ acting on the polynomial ring via permutation of variables and switching of signs was also considered \cite{G-K-R-2}. Recently, Debus and Riener considered $D_n$-invariant forms where $D_n$ is the subgroup of $B_n$ of even number of sign changes. All these groups have in common that they are \emph{reflection groups}. 

A finite group $G$ is a \emph{real reflection group} if $G \subset \operatorname{GL}_n(\R^n)$ is such that the matrix group is generated by \emph{reflections}, i.e. isometries $\R^n \to \R^n$ with a hyperplane as the set of fixed points. We usually just say that an abstract group $G$ is a real reflection group and the representation of $G$ is implicitly known. A real reflection group is called \emph{essential} if no non-trivial subspace of $\R^n$ is point wise fixed. It is a classical result that any real reflection group can be decomposed into a direct product of essential reflection groups. The essential real reflection groups were fully classified by Coxeter \cite{MR1503182,MR1581693}. There are four infinite families $A_n,B_n,D_n$ and $I_2(m)$ and six exceptional real reflection groups $H_3$, $H_4$, $F_4$, $E_6$, $E_7$, and $E_8$.

For $B_n,D_n$ and trivially $I_2(m)$ the equivariant classification of nonnegativity versus sums of squares was completed in \cite{Debus-Riener}. It is a natural question to consider the remaining infinite series of essential reflection groups $A_n$ and to study the equivariant nonnegativity versus sums of squares question. In this paper we initiate a study of $A_n$-invariant quartics. Although the vector space dimension of $A_n$-invariant quartics is only $2$, we will see that the understanding is challenging. A reason for the complexity involved here is that we do not consider nonnegativity of a polynomial globally. We consider nonnegativity on a hyperplane and do consider sums of squares modulo an ideal which is in general a very difficult problem.

The paper is structured as follows. Section \ref{sec:The reflection group of type $A$} explains the action of the group $A_n$ on an $n$-dimensional vector space and the induced action on the polynomial ring. Following this, we examine the sets of nonnegative and sums of squares $A_n$ invariant quartics in Section \ref{sec:SOSversusPSD}. We begin in Subsection \ref{subsection Nonnegativity vs U nonengativity} to elaborate on the difference between global nonnegativity of quartics and nonnegativity of $A_n$-invariant quartics. In Subsection \ref{subsection:PSD invariant} we provide the extremal elements of the cone of $A_n$-invariant nonnegative quartics before we analyse the $A_n$-invariant sums of squares quartics in Subsection \ref{subsection:SOS invariant}. Finally, we present a proof of our main theorem, Theorem \ref{thm:main} in Subsection \ref{subsection Proof of the theorem}. 

\section{The reflection groups of type $A$ and $A_n$-invariant polynomials} \label{sec:The reflection group of type $A$}

The real reflection group $A_n$ is, as a group, isomorphic to the symmetric group $S_{n+1}$. Recall that the reflection group $S_{n+1}$ is acting on $\R^{n+1}$ via permutation of coordinates in all possible ways. We call this action the \emph{permutation action} of the symmetric group. There is a non-trivial fixed subspace which is spanned by the vector $(1,\ldots,1)$ under the permutation action and thus  the permutation action does  not define an  \emph{essential} real reflection group. The action of $S_{n+1}$ on the invariant subspace $U_n:=\{ a \in \R^{n+1} : \sum_{i=1}^n a_i = 0\}$ via permutation of coordinates defines an essential real reflection group called $A_n$. We also say that it is the reflection group of \emph{type $A$}.

Recall that any group $G$ acting on $\R^n$ induces an action of $G$ on the polynomial ring $\R[\mathbf{x}]$ in $n$ variables. The action is as follows:
\[ \sigma \cdot f(\mathbf{x}) := f( \sigma^{-1} \cdot \mathbf{x})\]
where $\mathbf{x}=(\mathbf{x}_1,\ldots,\mathbf{x}_n)$ is a basis of the dual vector space of $\R^n$ and $\sigma \in G$. We refer to (\cite{Blek-Riener}, Section $4$) for details.

It is a classical result by Chevalley, Sheppard and Todd that the \emph{invariant ring} of real polynomials under the action of a finite matrix group in $\operatorname{GL}_n(\R)$ is isomorphic to a polynomial ring if and only if the group is a real reflection group \cite{MR0072877,MR0059914}.

In order to study $A_n$ invariant forms we consider the restriction of the permutation action of the symmetric group $S_{n+1}$ to the $n$ dimensional real vector space \[U_n =\left\{ a \in \R^{n+1} : \sum_i a_i = 0\right\}  .\]
Let $e_i \in \R^n$ denote the unit vector with $1$ at the $i$-th coordinate. A linear basis of $U_n$ is \[u_1=e_1-e_2,\ldots,u_n=e_1-e_{n+1}.\] The group $A_n$ acts on $U_n$ via permutation of the $e_i$'s in all possible ways. We obtain an induced action on an $n$-variate polynomial ring $\R[\mathbf{y}]$, where $\mathbf{y}$ is a basis of the dual vector space of $U_n$ and on the quotient of an $(n+1)$-variate polynomial ring $\R[\mathbf{x}]$ modulo the ideal generated by the linear polynomial $\mathbf{x}_1+\ldots+\mathbf{x}_{n+1}$. While $A_n$ does act on the $(n+1)$-variate quotient ring $\R[\mathbf{x}]/(\mathbf{x}_1+\ldots+\mathbf{x}_{n+1})$ via permutation of the $\mathbf{x}_i$'s, the reflection group does not permute the $\mathbf{y}_i$'s. We recall that those rings are isomorphic and two equivalent representations of $A_n$.

For two real representations $V,W$ of a group $G$ we say a linear map $\phi: V \to W$ is \emph{$G$-equivariant} if $\sigma \cdot \phi (v) = \phi (\sigma \cdot w)$ for any $v \in V, w \in W, \sigma \in G$.

\begin{prop}
The ring homomorphism $\R[\mathbf{y}] \to \R[\mathbf{x}]/(\mathbf{x}_1+\ldots+\mathbf{x}_{n+1})$ defined by $ \mathbf{y}_i \mapsto \mathbf{x}_1 - \mathbf{x}_{i+1}$, for all $1 \leq i \leq n$,
is a $A_n$-equivariant isomorphism.    
\end{prop}
\begin{proof}
Recall that $A_n$ fixes the subspace defined by $\mathbf{x}_1+\ldots+\mathbf{x}_{n+1} = 0$. A basis of this subspace is $\mathbf{x}_1-\mathbf{x}_{i+1}$ for $1 \leq i \leq n$. The basis elements and $\mathbf{x}_1+\ldots+\mathbf{x}_{n+1}$ form a basis of the degree $1$ part of $\R[\mathbf{x}]$. We have 
\[ \R[\mathbf{x}] \quad \cong \quad \R[\mathbf{x}_1-\mathbf{x}_2,\ldots,\mathbf{x}_1-\mathbf{x}_{n+1}][\mathbf{x}_1+\ldots+\mathbf{x}_{n+1} ] \]
and
\[ \R[\mathbf{x}]/(\mathbf{x}_1+\ldots+\mathbf{x}_{n+1}) \quad \cong \quad \R[\mathbf{y}] \]
With the discussion above the induced linear isomorphism is $A_n$-equivariant.
\end{proof}

Since we have a ring isomorphism we have that being a sum of squares is equivalent for the image and preimage. Moreover, nonnegativity of the preimage is equivalent to nonnegativity of the image on the subspace $U_n$ of $\R^{n+1}$. 

We denote by $p_k$ the \emph{power sum} polynomial of degree $k$ in the $(n+1)$-variables $\mathbf{x}$, i.e. $p_k = \sum_{i=1}^{n+1} \mathbf{x}_i^k$. It is classically known that the power sum polynomials $p_2,\ldots,p_{n+1}$ generate the $A_n$-invariant ring as $\R$-algebra modulo the ideal $(p_1)$. 

\begin{thm}
The invariant ring of $A_n$ is isomorphic to a polynomial ring. The invariant ring of $A_n$ acting via permutation of the variables $\mathbf{x}$ on the 
$(n+1)$-variate quotient ring $\R[\mathbf{x}]/(p_1)$ is 
$ \R[\mathbf{x}]^{A_n} \simeq \R[p_2,\ldots,p_{n+1}]\, .$ 
\end{thm}

\section{SOS versus PSD for $A_n$-invariant Quartics} \label{sec:SOSversusPSD}
In this Section we prove our main result Theorem \ref{thm:main}. We mainly restrict our notation and definitions to quartics.
Since the invariant ring is generated by the power sums $p_2,\ldots,p_{n+1}$ the vector space of \emph{$A_n$-invariant quartics} is $2$ dimensional and is spanned by the quotient classes of $p_2^2$ and $p_4$.

\begin{definition} \label{def: def of a psd}
We call a $A_n$-invariant quartic in $\R[\mathbf{x}]/(p_1)$ \emph{nonnegative} or \emph{psd} if and only if any element in its quotient class in $\R[\mathbf{x}]$ is nonnegative on $U_n$. We denote the set of psd $A_n$-invariant quartics by $P^{A_n}$. We call a $A_n$-invariant quartic in  $\R[\mathbf{x}]/(p_1)$ a \emph{sum of squares} or \emph{sos} if and only if an element in its quotient class in $\R[\mathbf{x}]$ is of the form $g_1^2+\ldots+g_m^2+p_1\cdot g$ for some $g_1,\ldots,g_m,g \in \R[\mathbf{x}]$. We denote the set of all $A_n$-invariant sos quartics by $\Sigma^{A_n}$.
\end{definition}

Suppose $f_1=ap_2^2+bp_4 + p_1 \cdot g_1 $ and $f_2=ap_2^2+bp_4+p_1\cdot g_2$ are two equivalent $A_n$-invariant quartics. Then nonnegativity of the quotient class $f_1 \mod (p_1)$ is well defined since $p_1=0$ on $U_n$.

The sets $P^{A_n}, \Sigma^{A_n}$ are pointed closed convex cones in the vector space $\R[\mathbf{x}]/(p_1)$.

The main result is the following.
  \begin{thm} \label{thm:main}
For $n \geq 3$ we have $P^{A_n}= \Sigma^{A_n}$ if and only if $n$ is odd. 
\end{thm}

 Note, we have $P^{A_n}= \Sigma^{A_n}$ by Hilbert's classification for all $n \leq 3$. 
 We will provide a proof of Theorem \ref{thm:main} in Subsection \ref{subsection Proof of the theorem}. Our strategy is as follows. First, we calculate the extremal rays of the two-dimensional cone $P^{A_n}$. Second, we give a description of $\Sigma^{A_n}$ using symmetry reduction. Third, we show that when $n$ is even then one of the extremal rays is not a sum of squares, while for odd $n$ both extremal rays are sum of squares.

 To motivate the fundamental difference between $S_n$-invariant and $A_n$-invariant nonnegative quartics we start with an overview on nonnegativity in Subsection \ref{subsection Nonnegativity vs U nonengativity}.

\subsection{Global nonnegativity versus nonnegativity on $U_n$} \label{subsection Nonnegativity vs U nonengativity}
We motivate the subtle difference between globally nonnegative forms and forms nonnegative on $U_n$ in the vector space $\langle p_2^2,p_4\rangle_\R$. For $n \geq 3$, the vector space of symmetric $(n+1)$-variate quartics is five dimensional and is spanned by the following products of power sum polynomials $$p_1^4,  p_2p_1^2, p_3p_1,p_2^2,p_4 \, .$$  For any $n \geq 3$, there exist $(n+1)$-variate symmetric quartic psd forms that are not sums of squares \cite{G-K-R}. For instance, there exists the following uniform example \cite{Acevedo-1} \[\mathfrak{f}_n:=4p_1^4-5p_2p_1^2-\frac{139}{20}p_3p_1+4p_2^2+4p_4\] which is always nonnegative but never a sum of squares for any number of variables $ \geq 4$. Note however, that restricting to the subspace $U_n$ gives $4(p_2^2+p_4)$. Thus $\mathfrak{f}_n$ is a sum of squares modulo the ideal $(p_1)$. The form $\mathfrak{f}_n$ can therefore not be used as a counter example for the reflection groups of type $A$.

We show that any psd form in the vector space $\langle p_2^2,p_4\rangle_\R$ is a sum of squares. The proposition follows also from the nonnegativity versus sums of squares classification in type $B$ \cite{G-K-R-2}. The quartics result for type $B$ was first observed by Choi, Lam and Reznick.

\begin{prop} \label{prop:symm4}
    Let $f=ap_2^2+bp_4$ be a nonnegative $(n+1)$-ary symmetric form, where $a,b \in \R$. Then $f$ is a sum of squares. 
\end{prop}

\begin{proof}
Since $f$ is an even symmetric form, nonnegativity of $f$ is equivalent to nonnegativity of $ap_1^2+bp_2$ on the probability simplex $\Delta_{n}:= \{ x \in \R_{\geq 0}^{n+1} : \sum_{i=1}^{n+1} x_i =1\}$. Then $p_1=1$ and $\frac{1}{n+1} \leq  p_2 \leq 1$ \cite{Acevedo-2}. In particular, we need to distinguish three cases depending on the sign of $a$. We can suppose $p_2=1$ and $\frac{1}{n+1}\leq p_4 \leq 1$. 
\begin{enumerate}
    \item If $a = 0$ we have $bp_4 $ nonnegative implies $b \geq 0$ and thus we have a sum of squares.
    \item If $a > 0$ we suppose without loss of generality that $a=1$ and we have $1+bp_2 \geq 0$ on $\Delta_{n}$ which implies $b \geq -1$. However, the form $$p_{2}^2-p_4 = p_4+2\sum_{i < j}\mathbf{x}_i^2\mathbf{x}_j^2-p_4=2\sum_{i < j}\mathbf{x}_i^2\mathbf{x}_j^2$$ on the boundary of the psd cone is clearly sos.
    \item If $a < 0$ we suppose $a=-1$ and have $-1+bp_2 \geq 0$ on $\Delta_{n}$ implies $b \geq n+1$. The form $(n+1)p_4-p_2^2$ on the boundary of the psd cone is a sum of squares since
    \begin{align*}
      (n+1)p_4-p_2^2 =  np_4-2\sum_{i < j}\mathbf{x}_i^2\mathbf{x}_j^2  = \sum_{i < j}(\mathbf{x}_i^2-\mathbf{x}_j^2)^2.
    \end{align*}
\end{enumerate}
\end{proof}

This subtle but important difference of nonnegativity on $\R^{n+1}$ and on $U_n$ has important structural consequences regarding $A_n$-invariant sums of squares.

\subsection{PSD $A_n$ invariant quartics} \label{subsection:PSD invariant}
A symmetric $(n+1)$-variate polynomial which is nonnegative on the linear subspace $U_n$ must not necessarily be globally nonnegative (see e.g. the polynomial $G_n$ for any $n$ and $F_n$ for any even $n$ in Lemma \ref{lem:extr A psd}). 
Since we are considering homogeneous invariant polynomials we have by biduality of convex cones (\cite{blekherman2012semidefinite}, Lemma 4.18.) the following Lemma.
\begin{lemma}
The boundary of $P^{A_n}$ consists of the forms $f=a \cdot p_2^2+b \cdot p_4$ for which there exists $0 \neq z \in U_n$ such that $f(z) = 0$.
\end{lemma}

In analogy to the proof of Proposition \ref{prop:symm4} we will analyse the maximum and the minimum of $p_4$ on the semialgebraic set defined by $p_2=1$ and $p_1=0$. 
 
\begin{lemma} \label{lem:extr A psd}
The extremal $(n+1)$-ary $A_{n}$-invariant psd quartics are
$$G_n := p_{2,2}-\frac{1}{\beta}p_4 \mbox{ and } F_n:= -p_{2,2}+\frac{1}{\alpha}p_4,$$ where \ $\displaystyle \beta = \frac{1-n+n^2}{n+n^2}$ and $\alpha = \left\{%
\begin{array}{ll}
  \displaystyle  \frac{1}{n+1} & \mbox { if } \  n \ \mbox{is odd}, \\
\displaystyle  \frac{4+2n+n^2}{2n+3n^2+n^3} & \mbox { if }  \ n \ \mbox{is even}.  \\
\end{array}\right. $
\end{lemma}

It follows directly from Proposition \ref{prop:symm4} that all of these extremal forms, but $F_n$ when $n$ is odd, are not globally nonnegative and cannot be a sum of squares in the polynomial ring $\R[\mathbf{x}]$.

\begin{proof}[Proof of Lemma \ref{lem:extr A psd}]
Since the quartics are homogeneous it is sufficient to analyse the minimum and maximum value of $p_4$ on $\mathbb{S}^n \cap U_n$. We have $p_1 =0 $ and $p_2=1$. This translates to the polynomial optimization problems 
\begin{align*}
    \min_{x \in \R} \pm p_4 \\
     \text{s.t. }   \, p_1 = 0 \\
    \quad \quad p_2=1 
\end{align*}
 By a variant of Timofte's half degree principle \cite[Theorem 1.1]{Riener2012} the extremas are attained at a point with at most $2$ different coordinates. The equality constraints transfer to the two equations
\begin{align*}
    lt+(n+1-l)s \quad = & \quad 0 \\
    lt^2+(n+1-l)s^2 \quad = & \quad 1
\end{align*}
where $0 \leq l \leq n+1$ is an integer and $s,t \in \R$ are real numbers. We observe that $l \not \in \{0,n+1\}$ which implies $1 \leq l \leq n$. For given integers $l$ and $n$ the equations provide unique solutions for $s$ and $t$ up to sign. However, inserting the solution in $p_4$ is independent of the signs of the coordinates and we have
\[ p_4(\underbrace{t,\ldots,t}_{l \text{ times}},\underbrace{s,\ldots,s}_{n+1-l \text{ times}}) =  \frac{(n+1)^2-3l(n+1)+3l^2}{(n+1)(n+1-l)l}\]
For the claim on the extremality of $F_n$ we are left with verifying 
\[ 
 \min_{ 1 \leq l \leq n, l \in \mathbb{Z}} \frac{(n+1)^2-3l(n+1)+3l^2}{(n+1)(n+1-l)l} 
\quad = \quad  
 \left\{   \begin{array}{cc}
 \displaystyle \frac{1}{n+1} & \text{ if $n$ is odd} \\
 \displaystyle  \frac{4+2n+n^2}{2n+3n^2+n^3} & \text{ if $n$ is even}
\end{array} \right.
\]
 which we do in Lemmas \ref{lem:minimum1} and \ref{lem:minimum2}, and verify
 \[ 
 \max_{ 1 \leq l \leq n, l \in \mathbb{Z}} \frac{(n+1)^2-3l(n+1)+3l^2}{(n+1)(n+1-l)l} 
\quad = 
\frac{1-n+n^2}{n+n^2}
\]
to prove that $G_n$ is extremal. This is Lemma \ref{lem:max}.
\end{proof}

\subsubsection{Verification of the extremality of $G_n$}
\begin{lemma} \label{lem:max}
 For all $n \geq 3$ we have $ \displaystyle \max_{ 1 \leq l \leq n, l \in \mathbb{Z}} \frac{(n+1)^2-3l(n+1)+3l^2}{(n+1)(n+1-l)l} $ is attained at $l = 1$ and $l=n$, and equals $\displaystyle \frac{1-n+n^2}{n+n^2}$.
\end{lemma}
\begin{proof}
    We calculate 
    \begin{align*}
        \frac{(n+1)^2-3(n+1)+3}{(n+1)n} \quad \geq & \quad \frac{(n+1)^2-3l(n+1)+3l^2}{(n+1)(n+1-l)l} \\
        \Leftrightarrow (-1+l)(n-l)(n+1)^2 \quad \geq & \quad  0.
    \end{align*}
    Note, for $1 \leq l \leq n$ the inequality is tight when $l \in \{1,n\}$ and otherwise strict.
\end{proof}

\subsubsection{Verification of the extremality of $F_n$}

\begin{lemma}  \label{lem:minimum1}
We have
 \[ \displaystyle \min_{ 1 \leq l \leq n} \frac{(n+1)^2-3l(n+1)+3l^2}{(n+1)(n+1-l)l} \quad =  \quad \frac{1}{n+1}\] and
  \[ \displaystyle \min_{ 1 \leq l \leq n, l \in \mathbb{Z}} \frac{(n+1)^2-3l(n+1)+3l^2}{(n+1)(n+1-l)l}  \quad > \quad   \frac{1}{n+1}\]if and only if $n$ is even.  
\end{lemma}
\begin{proof}
    For $1 \leq l \leq n$ we have \begin{align*}
    & \frac{(n+1)^2-3l(n+1)+3l^2}{(n+1)(n+1-l)l} -\frac{1}{n+1}  & \geq 0 \\
    \iff &\frac{(n+1)^2-3l(n+1)+3l^2}{(n+1)(n+1-l)l} -\frac{(n+1-l)l}{(n+1)(n+1-l)l} &\geq 0 \\
    \iff &(n+1)^2-3l(n+1)+3l^2-(n+1-l)l &\geq 0 \\ 
    \iff &(n+1-2l)^2 &\geq 0
\end{align*}
The last inequality is tight on integer values $1 \leq l \leq n$ if and only if $n+1$ is even.
\end{proof}

\begin{lemma} \label{lem:minimum2}
    If $n$ is even,  then $ \displaystyle \min_{ 1 \leq l \leq n, l \in \mathbb{Z}} \frac{(n+1)^2-3l(n+1)+3l^2}{(n+1)(n+1-l)l} $ is attained at $\displaystyle l = \frac{n}{2}$ and $\displaystyle l=\frac{n}{2}+1$, and equals $\displaystyle \frac{4+2n+n^2}{2n+3n^2+n^3}.$
\end{lemma}
\begin{proof}
 Evaluating $\displaystyle \frac{(n+1)^2-3l(n+1)+3l^2}{(n+1)(n+1-l)l}$ at $ \displaystyle l=\frac{n}{2}$ and $\displaystyle l=\frac{n+2}{2}$ gives $\displaystyle \frac{4+2n+n^2}{2n+3n^2+n^3}$. 
 
 \noindent Moreover, the denominator of 
 \begin{align*}
     \frac{(n+1)^2-3l(n+1)+3l^2}{(n+1)(n+1-l)l} - \frac{4+2n+n^2}{2n+3n^2+n^3} = \frac{(n+1)(4l^2-4l(n+1)+n(n+2))}{l(n+1-l)n(n+2)}
 \end{align*}
 is strictly positive for all $1 \leq l \leq n$. The 
 numerator is also nonnegative since
 $$ 4l^2-4l(n+1)+n(n+2) = (2l-(n+1))^2-1 \geq 0$$ because $n+1$ is odd.
\end{proof}

\subsection{SOS $A_n$-invariant quartics} \label{subsection:SOS invariant}

Given the action of a reflection group, representation theory and invariant theory can be applied to effectively describe the invariant sums of squares cone. We briefly sketch the symmetry reduction for sums of squares invariant by a reflection group. More details can be found in \cite{Blek-Riener,Debus-Riener,GatPar03,Heaton-Hosten-Shankhar}.
A reflection group $G$ acts on the vector space $\R[\mathbf{x}]_{d}$ of all $(n+1)$-variate forms of degree $d$ giving it the structure of a $G$-module.  We can decompose every  $G$-module into a direct sum of its irreducible sub-modules to obtain its \textit{isotypic decomposition}. Given an isotypic decomposition one constructs a \textit{symmetry adapted basis}, which can be used to understand the invariant sums of squares of elements in $\R[\mathbf{x}]_{2d}$. We outline this in the following.

First, we note that there is a natural projection onto the invariant part of $\mathcal{F}_{n,d}$ via the so called \textit{Reynolds-Operator} of the group $G$: \[\mathcal{R}_G : \R[\mathbf{x}]_{d} \to \R[\mathbf{x}]_{d} ^G,\; f \mapsto \frac{1}{|G|}\sum_{\sigma \in G}\sigma f.\] 
Suppose that we have \[\R[\mathbf{x}]_{d}  \simeq \bigoplus_{j=1}^\ell \eta_j\mathcal{V}_j\] is the isotypic decomposition of the $G$ action on $\R[\mathbf{x}]_{d} $, i.e. $\mathcal{V}_j$ are pairwise non-isomorphic irreducible $G$-modules and each occurs with multiplicity $\eta_j \in \N$ in $\R[\mathbf{x}]_{d} $. A symmetry adapted basis is a list  \[\{f_{11},\ldots,f_{1\eta_1},f_{21},\ldots,f_{\ell\eta_\ell}\}\]with the property  that for every $j$ there are  $G$-equivariant homomorphisms $\phi_{ji}$  which map $f_{j1}$ to $f_{ji}$  for all $1\leq i\leq \eta_j$, and furthermore that the orbit of each $f_{ji}$ spans an irreducible $G$-module isomorphic to $\mathcal{V}_j$ and the set of all orbits of all $f_{ji}$ spans $\R[\mathbf{x}]_{d} $. Given a symmetry adapted basis we can construct matrix polynomials \[B_j := \left(\mathcal{R}_G(f_{ji_1}f_{ji_2}\right)_{1 \leq i_1,i_2 \leq \eta_j} \text{ for }1 \leq j \leq \ell.\]
With these notations we have the following (see \cite[Theorem 2.6]{Debus-Riener}):
\begin{prop} \label{prop:describing invariant sos}
Let $f\in\R[\mathbf{x}]_{2d} ^G$ be an invariant form. Then $f$ is a sum of squares if and only if there exists positive semidefinite matrices $A_1,\ldots,A_\ell$ such that \[f=\sum_{j=1}^\ell \operatorname{Tr}(A_jB_j),\] where the matrix polynomials $B_j$ are constructed from a symmetry adapted bases of $\R[\mathbf{x}]_{d} $ as defined above. 
\end{prop}

Note that calculating an isotypic decomposition  of $\R[\mathbf{x}]_{d} $ and a symmetry adapted basis can in principle be done with linear algebra (see \cite{serre}). For the case of finite groups Hubert and Bazan \cite{Hubert-Bazan} constructed an algorithm to calculate equivariants which allows for an effective determination of symmetry adapted basis for all degrees. In the case when $G \in \{A_{n-1},S_n,B_n,D_n\}$ so-called \textit{higher Specht polynomials} can be used and their construction is completely combinatorial \cite{Debus-Riener, Morita}.


We denote by $\mathbb{S}^\lambda$ the Specht module associated with a partition $\lambda$.
\begin{lemma}
For $n\geq 3$, the $S_{n+1}$ isotypic decomposition of $\R[\mathbf{x}]_{2}$ equals
  \begin{align*}
      \R[\mathbf{x}]_{2} \quad = \quad  & \langle p_2 \rangle_\R \oplus \langle p_1^2 \rangle_\R \oplus \langle p_1(\mathbf{x}_i-\mathbf{x}_j) : i < j \rangle_\R \\
      & \oplus \langle \mathbf{x}_i^2-\mathbf{x}_j^2 : i < j \rangle_\R \oplus \langle (\mathbf{x}_i-\mathbf{x}_j)(\mathbf{x}_k-\mathbf{x}_l) : \# \{i,j,k,l\} = 4\rangle \\
      & 2 \cdot \mathbb{S}^{(n)} \oplus 2 \cdot \mathbb{S}^{(n-1,1)} \oplus \mathbb{S}^{(n-2,2)}
  \end{align*}
\end{lemma}
The proof is fully computational and we calculate a symmetry adapted basis based on higher Specht polynomials \cite{Morita} (and refer to \cite{Debus-Riener} for details). 

We apply the Reynolds-Operator of the symmetric group $S_{n+1}$ to pairwise products of equivariants of the isotypic decomposition which do not use $p_1$, since we consider sum of squares in $\R[\mathbf{x}]$ modulo the ideal $(p_1)$. 
\begin{lemma} \label{lem:sym}
For $n \geq 4$, we have
    \begin{align*}
        \mathcal{R}_{S_{n+1}}(p_2^{2}) \quad = \quad & p_{2}^2 \ , \\
        \mathcal{R}_{S_{n+1}}((\mathbf{x}_1^2-\mathbf{x}_2^2)^2)\quad = \quad &\frac{2}{n}p_4-\frac{2}{(n+1)n}p_2^2 , \ \ \text{and}  \\
          \mathcal{R}_{S_{n+1}}((\mathbf{x}_1-\mathbf{x}_2)^2(\mathbf{x}_3-\mathbf{x}_4)^2) \quad = \\
   \quad &       \hspace{-3.1cm}  
           \frac{4 (p_1^4 + 3 p_2^2 - 4 p_3p_1 + (n+1)^2 (p_2^2 - p_4) +     (n+1) (-2 p_2p_1^2 - 3 p_2^2+ 4 p_3p_1 + p_4))}{(n+1)n(n-1)(n-2)}. 
    \end{align*}
\end{lemma}
\begin{proof}
We calculate the polynomials $p_2^{2}, (\mathbf{x}_1^2-\mathbf{x}_2^2)^2, (\mathbf{x}_1-\mathbf{x}_2)^2(\mathbf{x}_3-\mathbf{x}_4)^2$ and apply the Reynolds-Operator $\mathcal{R}_{S_{n+1}}$. We then obtain the right hand side of the equation in the basis of monomial symmetric polynomials. We then use the package Symmetric Function in Sage to obtain a representation in terms of the power sum polynomials.    
\end{proof}

\begin{lemma}
If $f \in \R[\mathbf{x}]$ is a $A_n$-invariant sum of squares quartic modulo the ideal $(p_1)$ then 
    $$ f =  a (p_4-\frac{1}{n+1}p_2^2)+b((1-n+n^2)p_2^2-n(1+n)p_4) + p_1\cdot g$$ for some $a,b \geq 0$ and $g \in \R[x]$.
\end{lemma}
\begin{proof}
 Since $f \in \R[\mathbf{x}]$ is $A_n$-invariant, we can apply the Reynolds-Operator $\mathcal{R}_{A_n} = \mathcal{R}_{S_{n+1}}$ to $ g_1^2+\ldots+g_m^2$ and consider  $\mathcal{R}_{S_{n+1}}(g_1^2+\ldots+g_m^2) \mod p_1$ which has to be of the form 
 $$ \lambda_1 p_2^2 + \lambda_2 (p_4-\frac{1}{n+1}p_2^2) + \lambda_3((1-n+n^2)p_2^2-n(1+n)p_4)  $$ 
 for some $\lambda_1,\lambda_2,\lambda_3 \geq 0$, by Lemma \ref{lem:sym} and the discussion above. We have  $$(1-n)^2p_2^2 \quad = \quad  n(n+1) (p_4-\frac{1}{n+1}p_2^2)+ ((1-n+n^2)p_2^2-n(1+n)p_4) $$ which proves the claim.
\end{proof}

\subsection{Proof of Theorem \ref{thm:main}} \label{subsection Proof of the theorem}
We are ready to prove Theorem \ref{thm:main}
\begin{proof}[Proof of Theorem \ref{thm:main}]
There are three statements that we want to show. First, the polynomial $G_n \in \R[\mathbf{x}]$ is a sum of squares modulo $(p_1)$ for all $n \geq 3$. Second, for $n \geq 4$ odd the polynomial $F_n \in \R[\mathbf{x}]$ is a sum of squares modulo $(p_1)$. Third, for $n \geq 3$ even the polynomial $F_n \in \R[\mathbf{x}]$ is not a sum of squares modulo $(p_1)$

\begin{enumerate}
    \item We have 
    \begin{align*}
        G_n \quad  & = \quad p_2^2-\frac{n+n^2}{1-n+n^2}p_4 \\
        & = \quad \frac{1}{1-n+n^2}((1-n+n^2)p_2^2-n(1+n)p_4)
    \end{align*}
   which shows that $G_n$ is a sum of squares modulo $(p_1)$.
   \item If $n \geq 4$ is odd we have $F_n$ is a sum of squares by Proposition \ref{prop:symm4}. This is, because $\alpha$ equals the global minimum of $p_4$ on $p_2=1$ and we have seen that the corresponding polynomial is a sum of squares.
    \item For even $n \geq 4$ we have $F_n = -p_2^2+\frac{2n+3n^2+n^3}{4+2n+n^2}p_4$. We suppose that $F_n$ is a sum of squares modulo $(p_1)$. We must have 
\begin{align*}
    -p_2^2+\frac{2n+3n^2+n^3}{4+2n+n^2}p_4 \quad = \quad a (p_4-\frac{1}{n+1}p_2^2)+b((1-n+n^2)p_2^2+(-n-n^2)p_4) 
\end{align*}
for some $a,b \geq 0.$
Comparing the coefficients implies
 \begin{align*}
 b \quad = \quad & -\frac{4}{4 - 6 n + n^2 + n^4} 
 \end{align*}
which is a contradiction.
\end{enumerate}
\end{proof}
%
%
%

\begin{thebibliography}{11}


\bibitem {Acevedo-2} J. Acevedo, G. Blekherman, S. Debus, and C. Riener. \textit{The wonderful geometry of the vandermonde map}, arXiv:2303.09512, 2023.

\bibitem {Acevedo-1} J. Acevedo, G. Blekherman, S. Debus, and C. Riener. \textit{At the limit of the symmetric psd and sos cones}, in preparation.

%
%
%



\bibitem {blekherman2012semidefinite} G. Blekherman, P. A. Parrilo, R. R. Thomas. \textit{Semidefinite optimization and convex algebraic geometry}, (2012) SIAM 2012.


\bibitem {Blek-Riener} G. Blekherman and C. Riener. \textit{Symmetric nonnegative forms and sums of squares}, Discrete and Computational Geometry, 65:764–799, 2021.

\bibitem{blum1998complexity} L. Blum. \textit{Complexity and real computation}, Springer Science \& Business Media, 1998.

\bibitem{MR0072877} C. Chevalley. \textit{Invariants of finite groups generated by reflections}, Amer. J. Math., 77:778–782, 1955.


%
%
%
%

\bibitem{MR1503182} H. S. M. Coxeter. \textit{Discrete groups generated by reflections}, Ann. of Math. (2), 35(3):588–621, 1934.

\bibitem{MR1581693} H. S. M. Coxeter. \textit{The Complete Enumeration of Finite Groups of the Form R2i= (RiRj)kij = 1}
, J. London Math. Soc., 10(1):21–25, 1935.

%
%
%
%
%

\bibitem {Debus-Riener} S. Debus, C. Riener. \textit{Reflection groups and cones of sums of squares}, Journal of Symbolic Computation, 119:112–144, 2023. 


%


\bibitem {GatPar03} K. Gatermann, P. A. Parrilo. \textit{Symmetry groups, semidefinite programs, and sums of squares}, Journal of  Pure Applied Algebra ,192, (1 - 3), 95–128, 2004.


%
%


\bibitem {G-K-R} C. Goel, S. Kuhlmann, B. Reznick. \textit{On the Choi-Lam  analogue of Hilbert's 1888 theorem for Symmetric forms}, Linear Algebra and its Applications, 496, 114-120, 2016. 


\bibitem {G-K-R-2} C. Goel, S. Kuhlmann, B. Reznick. \textit{The analogue of Hilbert's 1888 theorem for Even Symmetric Forms}, Journal of Pure and Applied Algebra, 221, 1438-1448, 2017. 

%
%
%
%

\bibitem {Hilb_1} D. Hilbert. \textit{\"{U}ber die Darstellung definiter Formen als Summe von Formenquadraten}, Math. Ann., 32 (1888), 342-350;  Ges. Abh. 2, 154-161, Springer, Berlin, reprinted by Chelsea, New York, 1981.

%

\bibitem{Heaton-Hosten-Shankhar} A. Heaton, S. Hosten, and I. Shankar. \textit{Symmetry adapted Gram spectrahedra}, SIAM Journal on Applied Algebra and Geometry 5(1): 140-164, 2021.

\bibitem{Hubert-Bazan} E. Hubert and E. Rodriguez Bazan. \textit{Algorithms for fundamental invariants and equivariants of finite groups}, Mathematics of Computation 91(337): 2459-2488, 2022.


\bibitem {Marshall} M. Marshall. Positive Polynomials and Sum of Squares, Vol. 146, Mathematical Surveys and Monographs, AMS, 2008.
%

\bibitem{Morita} H. Morita and H.-F. Yamada. \textit{Higher Specht polynomials for the complex reflection group G (r, p, n)}, Hokkaido mathematical journal 27(3): 505-515, 1998. 

\bibitem{Motz-1} T. S. Motzkin. \textit{The arithmetic-geometric inequality}, in Inequalities, Oved Shisha (ed.) Academic Press, 205-224, 1967.



\bibitem {Riener2012} C. Riener. \textit{On the degree and half-degree principle for symmetric polynomials}, Journal of Pure and Applied Algebra, 216 (4), 850-856, 2012. 

%
%
%
%
%


\bibitem{serre} J.-P. Serre. \textit{Linear representations of finite groups}, Vol. 42. New York: Springer, 1977.

\bibitem{MR0059914} G. C. Shephard and J. A. Todd. \textit{Finite unitary reflection groups}, Canad. J. Math., 6:274–304, 1954.



\end{thebibliography}
%

\end{document}